\documentclass[11pt]{amsart}

\title[Twist spun knots of twist spun knots]{Twist spun knots of twist spun knots of classical knots}

\author[Fukuda]{Mizuki Fukuda}
\address{Mathematics for Advanced Materials Open Innovation Laboratory, AIST, c/o AIMR, Tohoku University, 2-1-1, Katahira, Aoba-ku, Sendai, Miyagi 980-8577, Japan}
\email{mizuki.fukuda.d2@tohoku.ac.jp}

\author[Ishikawa]{Masaharu Ishikawa}
\address{Department of Mathematics, Hiyoshi Campus, Keio University, 
4-1-1 Hiyoshi, Kohoku, Yokohama 223-8521, Japan}
\email{ishikawa@keio.jp}

\usepackage{amsfonts,amsmath,amssymb,amscd}
\usepackage{amsthm}
\usepackage{latexsym}
\usepackage{graphicx}
\usepackage{psfrag}
\usepackage{color}
\usepackage{ulem}

\usepackage{fancybox, ascmac}
\usepackage{multicol}

\usepackage{mathdots}

\theoremstyle{plain}
\newtheorem*{theorem*}{Theorem}
\newtheorem*{lemma*} {Lemma}
\newtheorem*{corollary*} {Corollary}
\newtheorem*{proposition*}{Proposition}
\newtheorem*{conjecture*}{Conjecture}
\newtheorem{theorem}{Theorem}[section]
\newtheorem{lemma}[theorem]{Lemma}

\theoremstyle{remark}

\newtheorem{remark}[theorem]{Remark}

\newtheorem*{example*}{Example}

\theoremstyle{definition}

\newtheoremstyle{citing}
  {}
  {}
  {\itshape}
  {}
  {\bfseries}
  {.}
  {.5em}
  {\thmnote{#3}}

\theoremstyle{citing}

\newcommand{\R}{\mathbb{R}}

\makeatletter
\@addtoreset{equation}{section}

\makeatother

\def\ees{{\accent"5E e}\kern-.385em\raise.55ex\hbox{\char'23}\kern-.08em}
\def\EES{{\accent"5E E}\kern-.5em\raise.8ex\hbox{\char'23}}
\def\ow{o\kern-.42em\raise.82ex\hbox{
\vrule width .12em height .0ex depth .075ex \kern-0.16em \char'56}\kern-.07em}
\def\OW{O\kern-.460em\raise1.36ex\hbox{
\vrule width .13em height .0ex depth .075ex \kern-0.16em \char'56}\kern-.07em}

\begin{document}

\begin{abstract}
A $k$-twist spun knot is an $n+1$-dimensional knot in the $n+3$-dimensional sphere which is obtained from an $n$-dimensional knot in the $n+2$-dimensional sphere by applying an operation called 
a $k$-twist-spinning.
This construction was introduced by Zeeman in 1965. 
In this paper, we show that the $m_2$-twist-spinning of the $m_1$-twist-spinning of a classical knot is a trivial $3$-knot in $S^5$ if $\gcd(m_1,m_2)=1$. We also give a sufficient condition for 
the $m_2$-twist-spinning of the $m_1$-twist-spinning of a classical knot to be non-trivial.
\end{abstract}

\dedicatory{Dedicated to Professor~Ti\ees n-S\ow n Ph\d{a}m on the occasion of his 60th birthday.}

\maketitle

\section{Introduction}

To study fibered links in higher dimensional spheres geometrically, one idea is to relate them to lower dimensional fibered links. In singularity theory, a cyclic suspension, or more generally the Thom-Sebastiani construction, is quite common. It enables us to study the Milnor fibration of the singularity $(f,0)$ of a polynomial map $f:\mathbb C^{n+1}\to \mathbb C$ at the origin $0$ in the form $f({\bf z}, w)=g({\bf z})+w^k$, $({\bf z},w)\in\mathbb C^n\times \mathbb C$, by using the Milnor fiberation of the lower dimensional singularity $(g,0)$.
A topological construction of the cyclic suspension is studied by Kauffman in~\cite{Kau74}.
His construction can be applied for simple fibered links in the odd dimensional spheres.
The fibered knot in $S^{2n+1}$ obtained by a cyclic suspension 
is a cyclic branched cover of $S^{2n-1}$ along the lower dimensional fibered knot
and the fiber surface is obtained correspondingly.
The monodromy matrix is obtained from the monodromy matrix of the lower one by taking a certain tensor product. 

In the study on higher dimensional knots, in 1965, Zeeman introduced a way of constructing an $n$-dimensional knot in $S^{n+2}$ from an $n-1$-dimensional knot in $S^{n+1}$ for $n\geq 2$~\cite{Zee65}.
This construction is called a {\it $k$-twist-spinning}. The knot $K^n$ in $S^{n+2}$ obtained from a knot $K^{n-1}$ in $S^{n+1}$ by the $k$-twist-spinning is called a {\it $k$-twist spun knot} of $K^{n-1}$.  
If $k\geq 1$, then a $k$-twist spun knot is a fibered knot and the fiber surface is the $k$-fold cyclic branched cover of $S^{n+1}$ along $K^{n-1}$ with removing one open ball. 
Moreover, the monodromy is periodic, shifting the sheets of the cyclic branched cover by $1$.
Properties of fibered knots obtained by a cyclic suspension and a $k$-twist-spinning are similar, though no relations between them are known yet.


Recently, the authors studied a more general class of $k$-twist spun knots in $S^4$ called a {\it branched twist spin}, and proved that, in most cases, branched twist spins are equivalent if and only if the corresponding  $1$-dimensional knots in $S^3$ are equivalent~\cite{FI23}. Here two knots $K_1$ and $K_2$ are said to be {\it equivalent} if there exists a diffeomorphism $\phi$ from the sphere containing $K_1$ to the sphere containing $K_2$ satisfying  $\phi(K_1)=K_2$. 
The key observation in that paper is that the quotient group of the fundamental group  of the complement of a branched twist spin of a $1$-dimensional knot $K$ in $S^3$ by its center is isomorphic to the fundamental group $\pi_1^{orb}(\mathcal O(K,m))$ of the $3$-orbifold $\mathcal O(K,m)$ of cyclic type with underlying space $S^3$ and ramification locus $K$ of order $m$ in most cases.

In this paper, we study a twist spun knot of a twist spun knot.
Let $m_1$ and $m_2$ be positive integers and $\tau_{m_2}(\tau_{m_1}(K))$ be the $m_2$-twist spun knot of the $m_1$-twist spun knot of a knot $K$ in $S^3$, which is a $3$-knot in $S^5$. 

\begin{theorem}\label{thm01}
If either $K$ is trivial or $\gcd(m_1,m_2)=1$, then $\tau_{m_2}(\tau_{m_1}(K))$ is trivial.
\end{theorem}

It is shown in~\cite[Proposition~3.2]{FI23} that if $K$ is non-trivial and $m_1\geq 2$, then 
$\tau_{m_1}(K)$ is a non-trivial $2$-knot in $S^4$. 
Theorem~\ref{thm01} means that
even if $\tau_{m_1}(K)$ is non-trivial,  $\tau_{m_2}(\tau_{m_1}(K))$ becomes trivial 
if $\gcd(m_1,m_2)=1$. We do not know if the same observation works for higher dimensional cases, see Remark~\ref{rem34}.

The converse of the assertion in Theorem~\ref{thm01} holds under a certain condition.

\begin{theorem}\label{thm02}
Set $m=\gcd(m_1,m_2)$. If $K$ is non-trivial, $m\geq 2$, and the center of $\pi_1^{orb}(\mathcal O(K,m))$ is trivial, then  $\tau_{m_2}(\tau_{m_1}(K))$ is non-trivial.
\end{theorem}

Note that if $K$ is a non-trivial and non-torus knot and the $m$-fold cyclic branched cover of $S^3$ along $K$ is aspherical, then  the center of $\pi_1^{orb}(\mathcal O(K,m))$ is trivial.
See the first half of the proof of~\cite[Lemma~3.1]{FI23}. 
For example, if $K$ is a hyperbolic knot with $m\geq 3$ or a prime satellite 
knot with $m\geq 2$, then the center of $\pi_1^{orb}(\mathcal O(K,m))$ is trivial.
We will prove in Theorem~\ref{lemmatorus} that  if $K$ is a $(p,q)$-torus knot and $m$ does not divide $pq$ then the center of $\pi_1^{orb}(\mathcal O(K,m))$ is non-trivial.


The authors would like to thank Makoto Sakuma for helpful suggestions.
The second author would like to thank the organizers of the International Conference "Singularities and Algebraic Geometry" Tuy Hoa 2024 for their kind hospitality and the organization of the excellent conference.
The second author is supported by JSPS KAKENHI Grant Numbers JP23K03098 and JP23H00081, and
JSPS-VAST Joint Research Program, Grant number JPJSBP120219602.

\section{$k$-twist-spinning and branched twist spin}

In this section, we introduce the $k$-twist-spinning of an $n$-dimensional knot in $S^{n+2}$
and a branched twist spin in $S^4$.
An $n$-dimensional knot (or $n$-knot for short) is the image of an embedding of $S^n$ into $S^{n+2}$.
Throughout this paper, $\text{\rm Int\,}X$ denotes the interior of a topological space $X$,
$\partial X$ denotes the boundary of $X$, and
$\text{\rm Nbd}(Y;X)$ denotes a compact tubular neighborhood of a topological space $Y$ embedded in $X$.

\subsection{$k$-twist-spinning}

Let $K^n$ be an $n$-knot. Choose a point $p$ in $K^n$ and a coordinate neighborhood $g:\R^{n+2}\to S^{n+2}$ such that $g(0)=p$ and  $g^{-1}(K^n)=\R^n\times \{(0,0)\}$. Let $B^{n+2}$ be the unit ball in $\R^{n+2}$
and set $X^{n+2}=g(B^{n+2})$ and $X^n=g(B^{n+2}\cap (\R^n\times\{(0,0)\}))$.
The pair $(X^{n+2}, X^n)$ is an unknotted $n$-ball in the $n+2$-ball.
Let $Y^{n+2}$ be the closure of $S^{n+2}-X^n$ and $Y^n$ be the closure of $K^n\setminus X^n$.
If $K^n$ is not trivial, then $Y^n$ is a knotted $n$-ball relatively embedded in the $n+2$-ball $Y^{n+2}$.
Gluing $(X^{n+2}, X^n)$ and $(Y^{n+2}, Y^n)$ by the identity map from $\partial X^{n+2}$ to $\partial Y^{n+2}$,
we recover the $n$-knot $K^n$ in $S^{n+2}$. See Figure~\ref{Fig1}.

\begin{figure}
\begin{center}
\includegraphics[scale=0.8, bb=127 546 564 713]{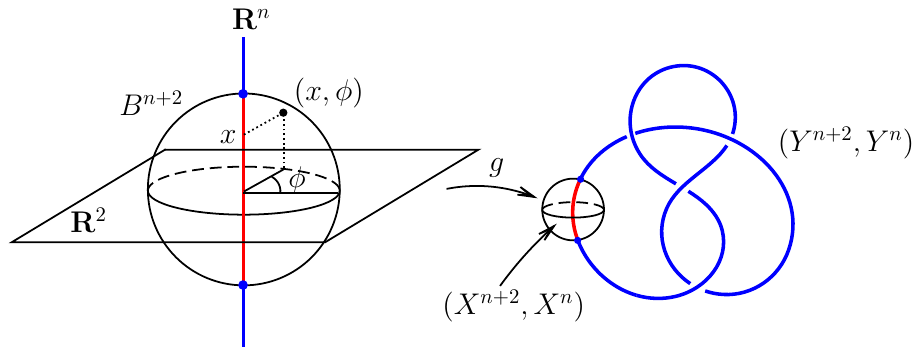}
\caption{A decomposition of $S^{n+2}, K^n)$ into $(X^{n+2}, X^n)$ and $(Y^{n+2}, Y^n)$,
and the coordinates $(x,\phi)$ on $X^{n+2}$.
The union $X^{n+2}\cup Y^{n+2}$ is $S^{n+2}$ and the union $X^n\cup Y^n$ is $K^n$.}\label{Fig1}
\end{center}
\end{figure}

Now set $X=(\partial X^{n+2}, \partial X^n)\times D^2$ and $Y=(Y^{n+2}, Y^n)\times \partial D^2$,
where $D^2$ is the unit disk in $\R^2$.
Let $x$ denote the coordinate of $g(\R^n\times \{(0,0)\})$ and $\phi$ denote the angular coordinate
of  $g(\{0\}\times (D^2\setminus\{(0,0)\}))$. Hence 
$X^{n+2}$ is equipped with the coordinates $(x,\phi)$,
where we think that the angle $\phi$ is not defined if $(x,\phi)\in g(\R^n\times\{(0,0)\})$.
Then we glue $X$ and $Y$ by the diffeomorphism $f:\partial X\to \partial Y$ defined by
\begin{equation}\label{eq21}
   ((x,\phi), \theta) \mapsto ((x, \phi+k\theta), \theta),
\end{equation}
where $(x,\phi)\in \partial X^{n+2}$, $\theta\in S^1=\partial D^2$ and $k\in\mathbb N\cup\{0\}$. 
The union $X\cup Y$ is the pair of the boundaries of 
$Y^{n+2}\times D^2$ and $Y^n \times D^2$, which is the pair of $S^{n+3}$ and an $n+1$-knot in $S^{n+3}$.
This $n+1$-knot is called the {\it $k$-twist spun knot} of $K^n$.
We denote it as $\tau_k (K^n)$.
Note that the above definition of $\tau_k (K^n)$ works even if  $k$ is negative. In this paper, we always assume that $k\geq 0$ by reversing the orientation of $S^4$ if necessary.

For a $1$-knot $K^1$, a presentation of $\pi_1(S^4\setminus \tau_k(K^1))$ can be obtained easily, see for example~\cite{Rol90, Zee65}.
Applying the same argument in these references to $\pi_1(S^{n+3}\setminus \tau_k(K^n))$,
we obtain the following presentation.
 
\begin{lemma}\label{lemma21}
Let $\langle x_1,\ldots, x_u\mid r_1,\ldots, r_v\rangle$ be a presentation of $\pi_1(S^{n+2}\setminus K^n)$ with generators $x_1,\ldots, x_u$ and relators $r_1,\ldots, r_v$.
Then 
\[
\pi_1(S^{n+3}\setminus \tau_k(K^n)) \cong 
\langle x_1,\ldots, x_u,\,h\mid r_1,\ldots, r_v,\, x_1hx_1^{-1}h^{-1}, \ldots, x_uhx_u^{-1}h^{-1},\, \mu^k h\rangle,
\]
where $\mu$ is a meridian of $K^n$ in $S^{n+2}$.
In particular, 
$\pi_1(S^{n+3}\setminus \tau_1(K^n))\cong\mathbb Z$ and
$\pi_1(S^{n+3}\setminus \tau_0(K^n))\cong \pi_1(S^{n+2}\setminus K^n)$.
\end{lemma}

\begin{proof}
The complement $S^{n+3}\setminus \tau_k(K^n)$ is the union of $M_1=(\partial X^{n+2}\setminus \partial X^n)\times D^2$ and $M_2=(Y^{n+2}\setminus Y^n)\times \partial D^2$. Since $(X^{n+2}, X^n)$ is unknotted, 
$\partial X^{n+2}\setminus \partial X^n$ is homotopic to $S^1$. 
Note that this homotopy is given by the deformation retract from $\partial X^{n+2}\setminus \text{\rm Int\,Nbd}(\partial X^n; \partial X^{n+2})$ to the equator of the $n+1$-sphere $\partial X^{n+2}$ 
shown on the left in Figure~\ref{Fig2}.
Hence  $S^{n+3}\setminus \tau_k(K^n)$ is homotopic to a manifold obtained from $M_2$ by attaching $S^1\times D^2$.
In particular, we have
\[
   \pi_1(S^{n+3}\setminus \tau_k(K^n))\cong \pi_1(M_2\cup D),
\]
where $D=\{*\}\times D^2\subset S^1\times D^2$.

\begin{figure}[htbp]
\begin{center}
\includegraphics[scale=0.8, bb=129 578 558 713]{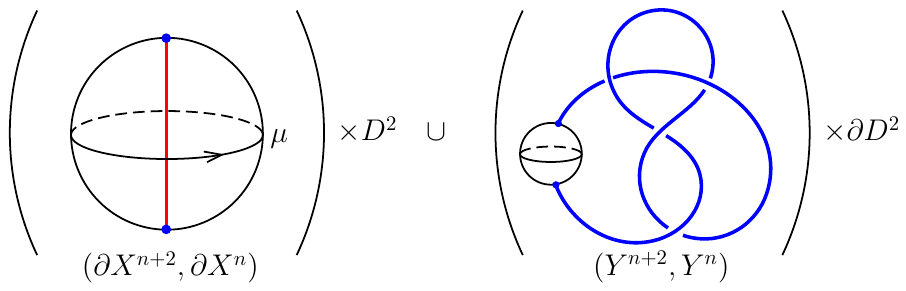}
\caption{A decomposition of $(S^{n+2}, K^n)$.
The union of $\partial X^{n+2}\times D^2$ and $Y^{n+2}\times \partial D^2$ is $S^{n+3}$ and the union of $\partial X^n\times D^2$ and $Y^n\times \partial D^2$ is $\tau_k(K^n)$.}\label{Fig2}
\end{center}
\end{figure}

The gluing map $f:\partial D\to \partial M_2$ of $D$ to $M_2$ is defined by~\eqref{eq21}.
Setting the coordinates on $\partial Y$ as $((\tilde x, \tilde\phi), \tilde\theta)$, we have
\[
   \tilde x=x,\quad \tilde \phi=\phi+k\theta,\quad \tilde \theta=\theta.
\]

Let $h$ be a loop given by $\{*\}\times \partial D^2\subset 
(\partial X^{n+2}\setminus \partial X^n)\times D^2=M_1$, 
which coincides with the loop $\{*\}\times \partial D^2\subset 
(Y^{n+2}\setminus Y^n)\times \partial D^2=M_2$.
This loop $h$ commutes any elements in $M_1$ and $M_2$. Therefore, it commutes any elements in
$\pi_1(S^{n+3}\setminus \tau_k(K^n))$. 
This loop is parametrized by $\tilde\theta\in S^1$ in $\partial M_2$.

Let $\mu$ denote a loop in $\partial X^{n+2}\setminus\partial X^n$ given by the parameter $\phi$.
This loop is a meridian of $K^n$ in $S^{n+2}$, which can be seen from the figure on the left in Figure~\ref{Fig2}.
Since $\partial D$ is parametrized by $\theta\in S^1$, 
its image in $\partial M_2$ rotates $k$ times along $\tilde\phi\in S^1$ and once along $\tilde\theta\in S^1$.
Since $\partial D$ is null-homotopic in $M_2\cup D$, we have $\mu^k h=1$. 
Thus the presentation of $\pi_1(S^{n+3}\setminus \tau_k(K^n))$ in the assertion is obtained.

If $k=1$ then the meridian $\mu=h^{-1}$ is in the center of $\pi_1(S^{n+3}\setminus \tau_1(K^n))$.
We can choose the generators of $\pi_1(S^{n+3}\setminus \tau_1(K^n))$ to be conjugate to $\mu$ (cf.~\cite{Kam01}). 
Since $H_1(S^{n+3}\setminus \tau_1(K^n))$ is isomorphic to $\mathbb Z$,  $\pi_1(S^{n+3}\setminus \tau_1(K^n))$ is the infinite cyclic group generated by $\mu$.

The assertion in the case $k=0$ is obtained from the presentation of $\pi_1(S^{n+3}\setminus \tau_k(K^n))$ immediately.
\end{proof}

\begin{theorem}[Zeeman \cite{Zee65}]\label{thmzeeman}
If $k\geq 1$ then
there exists a locally trivial fibration 
\[
   f:S^{n+3}\setminus \text{\rm Int\,Nbd}(\tau_k(K^n); S^{n+3})\to S^1.
\]
The fiber is diffeomorphic to the $k$-fold cyclic branched cover of $S^{n+2}$ along $K^n$ with removing an open ball and the monodromy is periodic, shifting the sheets of the cyclic branched cover by $1$.
\end{theorem}

\begin{remark}\label{rem23}
If $k=1$, then the fiber is diffeomorphic to the $n+1$-dimensional unit ball.
Therefore, $\tau_k(K^n)$ is a trivial $n+1$-knot in $S^{n+3}$.
\end{remark}

\subsection{Branched twist spin}

Circle actions on homotopy $4$-spheres are classified by Montgomery and Yang~\cite{MY60} and Fintushel~\cite{Fin76}, and later, Pao proved in~\cite{Pa78} that such homotopy $4$-spheres are the standard $4$-sphere $S^4$. 
A $2$-sphere embedded in $S^4$ that is invariant under a circle action is called a {\it branched twist spin} (cf.~\cite[\S 16.3]{Hil02}). 
A branched twist spin is obtained from a knot $K$ in $S^3$ by ``twisting'', like a $k$-twist-spinning, 
but the ``twisting'' for a branched twist spin is defined by two integers $m$ and $n$,
where $(m,n)$ is either a pair of coprime integers in $\mathbb Z\times \mathbb N$ or $(0,1)$.
We may assume that $m\geq 0$ by reversing the orientation of $S^4$ if necessary.
In this paper, since we study branched twist spins without fixing the orientation of $S^4$,
we always assume that $m\geq 0$.
If $(m,n)=(k,1)$ then it is nothing but the $k$-twist spun knot of $K$.


In~\cite{Pa78}, Pao constructed a fibered $2$-knot by setting the $m$-fold cyclic branched cover of $S^3$ along $K$ with removing an open ball to be a fiber and the monodromy to be periodic, shifting the sheets of the cyclic branched cover by $n$, and then prove that the total space is the complement of a $2$-knot in $S^4$.
His results are summarized as follows:

\begin{theorem}[Pao \cite{Pa78}]\label{thmpao}
Let $\tau_{m,n}(K)$ be the branched twist spin of $K$ defined by 
a pair of coprime positive integers
$m$ and $n$,
where $K$ is a knot in $S^3$.
Then, there exists a locally trivial fibration
\[
   f:S^4\setminus \text{\rm Int\,Nbd}(\tau_{m,n}(K);S^4)\to S^1.
\]
The fiber is diffeomorphic to the $m$-fold cyclic branched cover of $S^3$ along $K$ with removing an open ball and the monodromy is periodic, shifting the sheets of the cyclic branched cover by $n$.
\end{theorem}

It is natural to ask when two branched twist spins are non-equivalent.
This problem is studied by the first author in~\cite{F17, F17a, F22}. See also~\cite{FI23, Hil23} for more recent results.

\section{Proofs of the theorems}

We first prove Theorem~\ref{thm01}.

\begin{proof}[Proof of Theorem~\ref{thm01}]
We prove that if either $K$ is a trivial knot or $m=1$ then $\tau_{m_2}(\tau_{m_1}(K))$ is a trivial $3$-knot in $S^5$. It is known that if $K^n$ in $S^{n+2}$ is trivial then $\tau_k(K^n)$ in $S^{n+3}$ is also trivial for any $k\geq 0$~\cite[Corollary~3]{Zee65}. Therefore, if $K$ is trivial then $\tau_{m_2}(\tau_{m_1}(K))$ is also.

Assume that $m=1$.
The fiber of  $\tau_{m_2}(\tau_{m_1}(K))$ is the $m_2$-fold cyclic branched cover $M_{m_2}(\tau_{m_1}(K))$ of $S^4$ along $\tau_{m_1}(K)$ with removing one open $4$-ball. We will prove that this is diffeomorphic to the unit $4$-ball.

As mentioned in Theorem~\ref{thmzeeman}, there exists a locally trivial fibration
\[
   f:S^4\setminus \text{\rm Int\,Nbd}(\tau_{m_1}(K);S^4)\to S^1
\]
whose fiber $F$ is the $m_1$-fold cyclic branched cover $M_{m_1}(K)$ of $S^3$ along $K$ with removing one open $3$-ball and its monodromy is the $1$-shift of the sheets of the cyclic branched cover.
The manifold $M_{m_2}(\tau_{m_1}(K))$ is obtained from the exterior $S^4\setminus \text{\rm Int\,Nbd}(\tau_{m_1}(K);S^4)$ by cutting it by a fiber $F$, making $m_2$ copies, gluing them in order, and filling the boundary by $S^2\times D^2$ canonically. 
The manifold before the final filling is the manifold whose fiber is $F$ and monodromy is the $m_2$-shift of the sheets of the cyclic branched cover. This is the exterior of a branched twist spin. 
Remark that the $(m_1,m_2)$-branched twist spin is defined for positive integers $m_1$ and $m_2$ satisfying $\gcd(m_1, m_2)=1$ or $(m_1,m_2)=(0,1)$. Hence the assumption $m=\gcd(m_1,m_2)=1$ is needed.
By filling the boundary of the manifold by $S^2\times D^2$, we obtain the manifold $M_{m_2}(\tau_{m_1}(K))$.
There are two way of filling $S^2\times D^2$. The difference is the so-called {\it Gluck twist}  (\cite{Glu62},  cf.~\cite{F18}).
In either case, this is nothing but the construction of a branched twist spin. 
In particular, by~\cite[Theorem 4]{Pa78},  $M_{m_2}(\tau_{m_1}(K))$ is diffeomorphic to $S^4$.
Since the fiber of  $\tau_{m_2}(\tau_{m_1}(K))$ is obtained from $M_{m_2}(\tau_{m_1}(K))$ by removing one open ball, it is diffeomorphic to the unit $4$-ball. Thus the $3$-knot $\tau_{m_2}(\tau_{m_1}(K))$, which is the boundary of the $4$-ball, is a trivial $3$-knot in $S^5$.
\end{proof}



Next, we show some lemmas for proving Theorem~\ref{thm02}.
Let $K$ be a knot in $S^3$ and 
$\langle x_1,\ldots, x_u\mid r_1,\ldots, r_v\rangle$ be a presentation of $\pi_1(S^3\setminus K)$ with generators $x_1,\ldots, x_u$ and relators $r_1,\ldots, r_v$. We may set the generators to be meridians.

\begin{lemma}\label{lemma31}
Let $\tau_{m_2}(\tau_{m_1}(K)))$ be the $m_2$-twist spun knot of the $m_1$-twist spun knot of a knot $K$ in $S^3$. Then 
\begin{equation}\label{eq31}
\begin{split}
&\pi_1(S^5\setminus \tau_{m_2}(\tau_{m_1}(K))) \\
& \quad \cong \langle x_1,\ldots, x_u,\,h_1, h_2\mid r_1,\ldots, r_v,\, x_1h_1x_1^{-1}h_1^{-1}, \ldots, x_uh_1x_u^{-1}h_1^{-1},\, \mu_1^{m_1} h_1, \\ 
& \hspace{41mm} \; x_1h_2x_1^{-1}h_2^{-1}, \ldots, x_uh_2x_u^{-1}h_2^{-1},\, h_1h_2h_1^{-1}h_2^{-1},\, \mu_2^{m_2} h_2
\rangle,
\end{split}
\end{equation}
where $\mu_1$ and $\mu_2$ are the meridians of $K$ and $\tau_{m_1}(K)$, respectively, and we can set them as $\mu_1=\mu_2=x_1$.
\end{lemma}

\begin{proof}
By Lemma~\ref{lemma21},
\[
\begin{split}
\pi_1(S^4\setminus \tau_{m_1}(K)) \cong 
\langle x_1,\ldots, x_u,\,h_1 \mid & \;r_1,\ldots, r_v,\, x_1h_1x_1^{-1}h_1^{-1}, \ldots, x_uh_1x_u^{-1}h_1^{-1},\, \mu_1^{m_1} h_1 \rangle.
\end{split}
\]
Then, the presentation in the assertion is obtained by applying Lemma~\ref{lemma21} to this presentation again.

The decomposition of $S^4$ used to define the $k$-twist-spinning is as in Figure~\ref{Fig3}. We now describe the decomposition of $S^5$ to read off the meridian $\mu_2$ of $\tau_{m_2}(\tau_{m_1}(K))$.
The first factor $(Y^4,Y^2)$ of the second piece of the decomposition of $S^5$ is obtained 
from $S^4$ by removing the interior of a $4$-dimensional ball $B^4$ intersecting $\tau_{m_1}(K)$ trivially.
We set this $B^4$ to be the product of the upper hemisphere $H$ of $\partial X^3$
and the second factor $D^2$ of $(\partial X^3,\partial X^1)\times D^2$.
Set $B^2$ to be the intersection of $B^4$ with $\tau_{m_1}(K)$. Note that $B^2=(H\cap \partial X^1)\times D^2$, where $H\cap \partial X^1$ is the north pole of the sphere on the left in Figure~\ref{Fig3}.
The first piece $(\partial X^4,\partial X^2)\times (D^2)'$ of the decomposition of $S^5$ is $(\partial B^4,\partial B^2)\times (D^2)'$, where we used the notation $(D^2)'$ instead of $D^2$ since it is different from $D^2$ in Figure~\ref{Fig3}. The meridian $\mu_2$ of  $\tau_{m_2}(\tau_{m_1}(K))$
is a loop on $(\partial B^4\setminus \partial B^2)\times \partial (D^2)'$ that goes around $\partial B^2\times \partial (D^2)'$ once and does not go around $\partial (D^2)'$. Hence it is a loop on $H\setminus \partial X^1$ that goes around $H\cap \partial X^1$ once.
This is nothing but the loop $\mu_1$ in Figure~\ref{Fig3}. As explained 
in the proof of Lemma~\ref{lemma21}, $\mu_1$ is a meridian of $K$.
Hence we have $\mu_2=\mu_1=x_1$.
\end{proof}

\begin{figure}[htbp]
\begin{center}
\includegraphics[scale=0.8, bb=129 578 558 713]{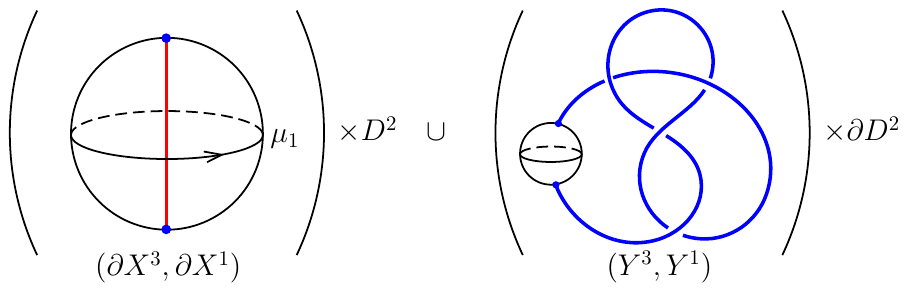}
\caption{A decomposition of $(S^4, \tau_{m_1}(K))$.}\label{Fig3}
\end{center}
\end{figure}

For a presentation $\langle x_1,\ldots, x_u\mid r_1,\ldots, r_v\rangle$ of $\pi_1(S^3\setminus K)$ of a knot $K$ in $S^3$, the fundamental group $\pi_1^{orb}(\mathcal O(K,m))$ of the orbifold $\mathcal O(K,m)$ with underlying space $S^3$ and remification locus $K$ is presented as
\begin{equation}\label{eq00}
   \pi_1^{orb}(\mathcal O(K,m))\cong\langle x_1,\ldots, x_u\mid r_1,\ldots, r_v,\, \mu^m\rangle
\end{equation}
where $\mu$ is a meridian of $K$. Using a Wirtinger presentation $\langle x_1,\ldots, x_u\mid r_1,\ldots, r_v\rangle$ of $\pi_1(S^3\setminus K)$, we may set $\mu=x_1$. 
See for instance~\cite{Thu77, DM84, Tak88, BP01} for basics of such orbifolds. 

\begin{lemma}\label{lemma33}
Suppose that $m=\gcd(m_1,m_2)\geq 2$ and the center of $\pi_1^{orb}(\mathcal O(K,m))$ is trivial.
Then, the center of $\pi_1(S^5\setminus \tau_{m_2}(\tau_{m_1}(K)))$ is isomorphic to $\mathbb Z$, generated by $x_1^m$, and the quotient group of $\pi_1(S^5\setminus \tau_{m_2}(\tau_{m_1}(K)))$ by its center is 
isomorphic to $\pi_1^{orb}(\mathcal O(K,m))$.
\end{lemma}


\begin{proof}
Take a Wirtinger presentation $\langle x_1,\ldots, x_u \mid\;  r_1,\ldots, r_v\rangle$ of $\pi_1(S^3\setminus K)$. 
Set $\mu_1=\mu_2=x_1$ and eliminate $h_1$ and $h_2$ from the presentation~\eqref{eq31} by using the relations $h_1=\mu_1^{-m_1}=x_1^{-m_1}$ and $h_2=\mu_2^{-m_2}=x_1^{-m_2}$ as
\[
\begin{split}
\pi_1(S^5\setminus \tau_{m_2}(\tau_{m_1}(K))) 
\cong \langle x_1,\ldots, x_u \mid\; & r_1,\ldots, r_v,\, x_2x_1^{-m_1}x_2^{-1}x_1^{m_1}, \ldots, x_ux_1^{-m_1}x_u^{-1}x_1^{m_1}, \\ 
& x_2x_1^{-m_2}x_2^{-1}x_1^{m_2}, \ldots, x_ux_1^{-m_2}x_u^{-1}x_1^{m_2}
\rangle.
\end{split}
\]
This group is the quotient group of $\pi_1(S^3\setminus K)$ by its normal subgroup $N$ generated by $x_kx_1^{-m_1}x_k^{-1}x_1^{m_1}$ and $x_kx_1^{-m_2}x_k^{-1}x_1^{m_2}$ for $k=2,\ldots,u$.
Let $G$ denote $\pi_1(S^5\setminus \tau_{m_2}(\tau_{m_1}(K)))$ for simplicity.
Then the above isomorphism is denoted as $G\cong  \pi_1(S^3\setminus K)/N$.
Since $N$ is contained in the normal closure $\langle\langle x_1^m \rangle\rangle$ of the group $\langle x_1^m\rangle$ generated by $x_1^m$, we have
\[    
\pi_1^{orb}(\mathcal O(K,m))
\cong \pi_1(S^3\setminus K)/\langle\langle x_1^m \rangle\rangle
\cong (\pi_1(S^3\setminus K)/N)/\langle\langle x_1^m \rangle\rangle\cong G/\langle x_1^m\rangle.
\]
Here the last isomorphism follows from the fact that $\langle x_1^m\rangle$ is central and hence normal in $G$. 

Let $p:G\to G/\langle x_1^m\rangle$ be the projection and let $Z(G)$ denote the center of $G$. 
Since $p$ is a group epimorphism, $p(Z(G))$ is contained in the center of $G/\langle x_1^m\rangle$.
Then, the assumption of the triviality of the center of $\pi_1^{orb}(\mathcal O(K,m))$ implies that $p(Z(G))$ is trivial.
Thus we have $Z(G)=\langle x_1^m\rangle$.
\end{proof}

\begin{proof}[Proof of Theorem~\ref{thm02}]
Suppose that $K$ is non-trivial, $m=\gcd(m_1,m_2)\geq 2$, and the center of $\pi_1^{orb}(\mathcal O(K,m))$ is trivial.
By Lemma~\ref{lemma33}, the quotient group of $\pi_1(S^5\setminus \tau_{m_2}(\tau_{m_1}(K)))$ 
by its center is isomorphic to $\pi_1^{orb}(\mathcal O(K,m))$. 
Since the abelianization of  $\pi_1^{orb}(\mathcal O(K,m))$ is isomorphic to $\mathbb Z/m\mathbb Z$,
if $m\geq 2$ then $\pi_1^{orb}(\mathcal O(K,m))$ is non-trivial.
On the other hand, 
for a trivial $3$-knot $O$, $\pi_1(S^5\setminus O)$ is isomorphic to $\mathbb Z$ and hence the quotient group of $\pi_1(S^5\setminus O)$ by its center is trivial. Therefore, $\tau_{m_2}(\tau_{m_1}(K))$ is non-trivial.
\end{proof}



\begin{remark}\label{rem34}
From a knot $K$ in $S^3$ and a sequence of positive integers ${\bf m}=(m_1, m_2, \ldots, m_N)$,
we can obtain an $N+1$ knot in $S^{N+3}$ by applying $m_i$-twist-spinning for $i=1,2,\ldots,N$ to $K$ inductively. Let $\tau_{\bf m}(K)$ denote this $N+1$-knot in $S^{N+3}$.
A presentation of $\pi_1(S^{N+3}\setminus \tau_{\bf m}(K))$ can be obtained by the same way as in Lemma~\ref{lemma31}. 
Similar to Lemma~\ref{lemma33}, we can prove that 
if $m=\gcd(m_1,m_2,\ldots,m_N)\geq 2$ and the center of $\pi_1^{orb}(\mathcal O(K,m))$ is trivial,
then the quotient group of $\pi_1(S^{N+3}\setminus \tau_{\bf m}(K))$ by its center is 
isomorphic to $\pi_1^{orb}(\mathcal O(K,m))$.
Hence Theorem~\ref{thm02} holds 
in the higher dimensional cases also. On the other hand,
it is not clear if a theorem similar to Theorem~\ref{thm01} holds or not 
since we used the fact, proved by Pao, that a cyclic branched cover of $S^4$ along a twist spun knot is $S^4$ and its brancherd locus is a branched twist spin~\cite{Pa78},
but a similar statement is not known in the higher dimensional cases.
\end{remark}

\section{On the center of  $\pi_1^{orb}(\mathcal O(K,m))$ for a torus knot $K$}


Let $K$ be a $(p,q)$-torus knot, where $p,q\geq 2$ and $\gcd(p,q)=1$.
Let  $x$ and $y$ be the elements in $\pi_1(S^3\setminus K)$ corresponding to the preferred meridian-longitude pair of the standard torus on which $K$ lies.
The meridian of $K$ is given as $x^sy^r$, where $r$ and $s$ are integers satisfying $pr + qs = 1$.
By~\eqref{eq00}, the orbifold group $\pi_1^{orb}(\mathcal O(K,m))$ is presented as 
\begin{equation}\label{torusorb}
\pi_1^{orb}(\mathcal O(K,m)) \cong \langle x, y \mid  x^py^{-q},\,(x^sy^r)^{m} \rangle.
\end{equation}

\begin{theorem}\label{lemmatorus}
Let $K$ be a $(p,q)$-torus knot with $p,q \geq 2$ and ${\rm gcd}(p,q)=1$. 
For $m \geq 2$, if $m$ does not divide $pq$, then the center of  $\pi_1^{orb}(\mathcal O(K,m))$ has a non-trivial element.
\end{theorem}

Hence Lemma~\ref{lemma33} does not hold for these torus knots.

\begin{proof}
Let $A:\pi_1^{orb}(\mathcal O(K,m))\to \mathbb Z/m\mathbb Z$ be the abelianization map.
Since $x^p$, which is equal to $y^q$, corresponds to an element in the center of $\pi_1(S^3\setminus K)$, 
$x^p$ is an element in the center of $\pi_1^{orb}(\mathcal O(K,m))$.
It is suffice to show that $x^p$ is non-trivial in $\pi_1^{orb}(\mathcal O(K,m))$.
For $n\in\mathbb Z$, let $\bar n\in\mathbb Z/m\mathbb Z$ denote the image of $n$ by $A$.
Set $A(x) = \bar a$ and $A(y) = \bar b$,
where $a, b\in\mathbb Z$.
Since $x^sy^r$ is a meridian of $K$, 
$A(x^sy^r)\equiv \bar a s+\bar b r\equiv 1 \mod m$.
Since $pr+qs=1$, the solution of $a s+b r=1$ is given in the form 
$a=q-tr$ and $b=p+ts$ for $t\in\mathbb Z$.
Since $x^p=y^q$, we have
$\overline{p(q-tr)}=A(x^p)=A(y^q)=\overline{q(p+ts)}$. 
This and $pr+sq=1$ imply that $t=0\mod m$. 
Hence we have  $A(x) = \bar q$ and $A(y) = \bar p$.
Since $\bar p\bar q \ne \bar 0$ by the assumption,
we have $A(x^p) \ne \bar 0$.
Hence $x^p$ is non-trivial in $\pi_1^{orb}(\mathcal O(K,m))$.
\end{proof}


\end{document}